\documentclass[reqno,a4paper,12pt]{amsart}
\usepackage{graphicx, amsfonts, amsmath}
\usepackage{caption}
\usepackage{amsthm}
\usepackage{fullpage}
\usepackage{float}
\usepackage{url}
\newtheorem{theorem}{Theorem}[section]

\newtheorem{conjecture}[theorem]{Conjecture}
\newtheorem{question}[theorem]{Question}
\newtheorem{prop}[theorem]{Proposition}
\newtheorem{corollary}[theorem]{Corollary}
\newcommand{\Z}{\mathbb{Z}}

\newcommand{\N}{\mathbb{N}}
\title{On sets of integers which contain no three terms in geometric progression}
\author{NATHAN MCNEW}
\address{Department of Mathematics, Dartmouth College \\ Hanover, NH 03755 USA}
\email{nathan.g.mcnew@dartmouth.edu}

\date{}

\begin{document}

\begin{abstract}
The problem of looking for subsets of the natural numbers which contain no 3-term arithmetic progressions has a rich history. Roth's theorem famously shows that any such subset cannot have positive upper density.  In contrast, Rankin in 1960 suggested looking at subsets without three-term geometric progressions, and constructed such a subset with density about 0.719.  More recently, several authors have found upper bounds for the upper density of such sets. We significantly improve upon these  bounds, and demonstrate a method of constructing sets with a greater upper density than Rankin's set.  This construction is optimal in the sense that our method gives a way of effectively computing the greatest possible upper density of a geometric-progression-free set.  We also show that geometric progressions in $\Z/n\Z$ behave more like Roth's theorem in that one cannot take any fixed positive proportion of the integers modulo a sufficiently large value of $n$ while avoiding geometric progressions.
\end{abstract}
\maketitle

\medskip
\section{Background} \label{background}
Let $A$ be a subset of the positive integers.  A three-term arithmetic progression in $A$ is a progression $(a,a+b,a+2b) \in A^3$ with $b>0$, or equivalently, a solution to the equation $a + c = 2b$ where $a,b$ and $c$ are distinct elements of $A$.  We say that $A$ is free of arithmetic progressions if it contains no such progressions.  For any subset $A$ of the positive integers we denote by $d(A)$ its asymptotic density (if it exists) and its upper density by $\bar{d}(A)$.  

In 1952 Roth proved the following famous theorem \cite{Roth}. 
\begin{theorem}[Roth] If $A$ is a subset of the positive integers with $\bar{d}(A)>0$ then $A$ contains a $3$-term arithmetic progression.
\end{theorem}
In particular, Roth showed that for any fixed $\alpha > 0$ and sufficiently large $N$, any subset of the integers $\{1,\cdots,N\}$ of size at least $\alpha N$ contains a 3-term arithmetic progression.  We can also view Roth's result as a statement about arithmetic progressions (with 3 distinct elements) in the group of integers mod $N$. Namely, if we denote by $D(\Z/N\Z)$ the size of the largest subset of $\Z/N\Z$ free of arithmetic progressions, Roth's argument can be used to  show \cite{Roth2} $D(\Z/N\Z) = O\left(\frac{N}{\log \log N}\right)$. This result has been improved several times, (see \cite{Heath}, \cite{Bourgain}, \cite{Szemer}) with the current best result, due to Sanders \cite{Sanders}, \begin{equation}\label{Sanders}D(\Z/N\Z) = O\left(\frac{N(\log \log N)^5}{\log N}\right).\end{equation} Roth's theorem has since been generalized by Szemer\'{e}di to progressions of arbitrary length.

Arithmetic-progression-free sets have also been studied in the context of arbitrary abelian groups.  Meshulam \cite{Meshulam} generalized Roth's theorem to finite abelian groups of odd order, and recently Lev \cite{Lev} has extended Meshulam's ideas to arbitrary finite abelian groups, a result which will be needed later.  Let $D(G)$ be the size of the largest subset of the finite group $G$ free of 3-term arithmetic progressions.  If 2 happens to divide $|G|$, it is possible to find examples of arithmetic progressions with repeated terms by looking at elements of order 2, so we insist that our arithmetic progressions consist of 3 distinct elements.  We also denote by $c(G)$ the number of components of the abelian group $G$ when written in the invariant factor decomposition $G \cong \Z/k_1\Z \times \Z/k_2\Z \times \cdots \times \Z/k_t\Z$ where $k_1\mid k_2\mid \cdots \mid k_t$, and write $2G$ for the group $\{g + g: g \in G\}$.

\begin{theorem}[Lev] For any abelian group $G$, $D(G)$ satisfies
\[D(G)\leq \frac{2|G|}{c(2G)}.\]
\end{theorem}
Combining Lev's theorem with Sanders' bound as in \cite[Corollary 1.3]{Meshulam}  one obtains
\begin{corollary} \label{Lcor} For any abelian group $G$ for which $c(G) = c(2G)$
\[D(G) \ll \frac{|G|(\log \log |G|)^5}{(\log |G|)^{\frac{1}{2}}}.\]
\end{corollary}

In a completely analogous manner, one can consider three-term geometric progressions of integers of the form ($a,ak,ak^2$) with $k \in \mathbb{Q}$ and $k>1$ (equivalently, a solution to the equation $ac=b^2$ with distinct integers $a,b,c$) and seek sets of integers which are free of such progressions. It is somewhat surprising just how different the results are in this case. We can immediately see, for example, that the square-free integers, a set with asymptotic density $\frac{6}{\pi^2} \approx 0.61$, is free of geometric progressions. 

Unlike the difference of two terms in an arithmetic progression, the ratio between successive terms of a geometric progression of integers need not be an integer.  For example, the progression $(4,6,9)$ is a geometric progression with common ratio $\frac{3}{2}$.  While most of the existing literature on this problem is concerned only with the rational ratio case, we will also consider the problem restricted to integer ratios in the results that follow. 

The problem of finding sets of integers free of geometric progressions was first considered by Rankin \cite{Rankin} who showed

\begin{theorem}\label{rankinconstruction} There exists a set, $G^*_3$, free of geometric progressions with asymptotic density $  \frac{1}{\zeta(2)}\prod_{i>0}\frac{\zeta(3^i)}{\zeta(2\cdot 3^i)}>0.71974.$ (Here $\zeta$ is the Riemann zeta function.)
\end{theorem}
Since then, several papers have further investigated the problem of finding the largest possible sets which are geometric-progression-free.  We will use the notation
\begin{align*} \overline{\alpha} &= \sup\{\bar{d}(A):A \subset \N \text{ is free of geometric progressions}\}\\
\alpha &= \sup\{d(A): A \subset \N \text{ is free of geometric progressions and } d(A) \text{ exists}\}
\end{align*}
and the symbols $\overline{\beta}, \beta$ for the corresponding values when we restrict the problem to integer ratios.  Clearly $\alpha \leq \overline{\alpha}$.  Rankin's construction remains the best lower bound for $\alpha$.  Riddell \cite{Riddell} showed that $\overline{\alpha} \leq 6/7$. This bound was reproved by Beiglb\"{o}ck, Bergelson, Hindman, and Strauss \cite{BBHS}, who were unaware of Riddell's result.  Their result was an improvement of a bound obtained by Brown and Gordon \cite{BG} who, also unaware of Riddell's result, had shown that $\overline{\alpha}<0.86889$.  Nathanson and O'Bryant \cite{NO} combined these methods to show that $\overline{\alpha} < 0.84948$. (They also corrected an error in the calculations of \cite{NO} for progressions of length greater than 3.)

While the existing literature on geometric-progression-free sets has worked in greater generality, considering geometric progressions of length $k \geq 3$, we have chosen in this paper to focus on progressions of length 3.  While most of the methods and constructions in this paper generalize to arbitrary $k$, they don't, in general, appear to lead to a closed form expression in terms of $k$.  We look first at the problem of finding geometric progressions of the residues mod $n$ in Section \ref{modn} and find that any positive proportion of such residues will contain a geometric progression for large enough $n$.  

Section \ref{integer} considers sets of integers free of progressions with integral ratios and shows that one can construct such sets with substantially higher upper density than Rankin's set.  

Our main results are in Section \ref{bounds}, where we demonstrate an algorithm for effectively computing the value of $\overline{\alpha}$.  We use this method to  significantly improve the best known upper bound for $\overline{\alpha}$, from 0.84948 down to 0.772059, and to show that $\overline{\alpha}>0.730027$, showing for the first time that $\overline{\alpha}$ is greater than the density of Rankin's construction.  These results are also applied to $\overline{\beta}$, which we are able to compute to somewhat better precision.  Finally, in Section \ref{density} we investigate bounds for sets which have an asymptotic density.

\section{Geometric Progressions in $\Z/n\Z$} \label{modn}
In light of Rankin's result, showing that in contrast to the arithmetic progression case we can take sets of integers with positive density (even the majority of integers) and still avoid geometric progressions, it is somewhat surprising that this does not remain true when we look at the integers mod $n$.  We need, first, a proposition and a corollary of Lev's theorem for unit groups $(\Z/n\Z)^\times$.

\begin{prop} \label{subgp} Given any subgroup $H$ of an abelian group $G$
\[D(G) \leq \frac{|G|D(H)}{|H|}.\]
\end{prop}
\begin{proof}
If $A \subset G$ is free of arithmetic progressions, then for every coset $xH$ the set $A \cap xH$ is free of arithmetic progressions, and by the pigeonhole principle there exists at least one coset with \[|A \cap xH|=|x^{-1}A \cap H|\geq \frac{|A||H|}{|G|}.\]  Thus, since $(x^{-1}A)\cap H$ is a subset of $H$ free of arithmetic progressions,
\[|A|\leq \frac{|G||(x^{-1}A)\cap H|}{|H|} \leq \frac{|G| D(H)}{|H|}. \qedhere \]
\end{proof}
\begin{corollary} \label{unit} For any integer $n$ the group $(\Z/n\Z)^\times$ of units$\mod n$ satisfies
\[D((\Z/n\Z)^\times) \ll \frac{\varphi(n)(\log \log n)^5}{(\log n)^{\frac{1}{2}}}.\]
\end{corollary}
\begin{proof} When $G=(\Z/n\Z)^\times$ then $|G| = \varphi(n)$. Write $G = \Z/2\Z \times \cdots \Z/2\Z \times H$ where $H$ is free of copies of $\Z/2\Z$ in its invariant factor decomposition (and hence $c(H) = c(2H)$).  Let $n=p_1^{k_1}p_2^{k_2}\cdots p_{\omega(n)}^{k_{\omega(n)}}$ be the prime factorization of $n$.  
Since $(\Z/n\Z)^\times \cong (\Z/p_1^{k_1}\Z)^\times \times (\Z/p_2^{k_2}\Z)^\times \times \cdots \times (\Z/p_{\omega(n)}^{k_\omega(n)}\Z)^\times$ and each group $(\Z/p_j^{k_j}\Z)^\times$ contains at most one copy of $\Z/2\Z$ when $p_j \neq 2$ and at most two if $p_j = 2$, we find that $G$ has at most $\omega(n) + 1$ copies of $\Z/2\Z$.  Thus the subgroup $H$ has size at least $\frac{\varphi(n)}{2^{\omega(n)+1}}$.  Thus, using Corollary \ref{Lcor} and Proposition \ref{subgp},
\begin{align*} D(G) \leq \frac{|G| D(H)}{|H|} & \ll \frac{|G| |H|(\log \log |H|)^5}{|H|(\log |H|)^{\frac{1}{2}}} \\
&\leq \frac{\varphi(n)(\log\log n)^5}{(\log (\frac{\varphi(n)}{2^{\omega(n)+1}}))^{\frac{1}{2}}}. 
\end{align*}
Using the facts \cite[Theorem~323~and~Section~22.10]{HW} that $\lim \inf \frac{\varphi(n) \log \log n}{n} = e^{-\gamma}$ and $\omega(n) = O\left(\frac{\log n}{\log \log n}\right)$, we have

\begin{align*} \log\left(\varphi(n)/2^{\omega(n)+1}\right) &\geq \log\left(\frac{n}{e^\gamma \log \log n}(1+o(1))\right) - (\omega(n)+1)\log(2)   \\
& = \log n + O\left(\frac{\log n}{\log \log n}\right).
\end{align*}
So 
\[D(G) \ll \frac{\varphi(n)(\log \log n)^5}{(\log n)^{\frac{1}{2}}}. \qedhere \]
\end{proof}

\begin{theorem}\label{geomod}
Let $E(\Z/n\Z)$ denote the size of the largest possible subset of the residues $\mod n$ which does not contain a $3$-term geometric progression.  Then \[E(\Z/n\Z) \ll \frac{n(\log \log n)^5}{(\log n)^{1/2}}.\]
\end{theorem}
\begin{proof} For each $d\mid n$ let $R_d$ denote the set of $m \in \Z/n\Z$ such that $(m,n)=d$.  So, for example, when $d=1$, $R_1 = (\Z/n\Z)^\times$.  Note that $R_d$ can be viewed as $d(\Z/\frac{n}{d}\Z)^\times$ in the sense that each element of $R_d$ is uniquely representable as $d$ times a residue coprime to $\frac{n}{d}$.  Thus $|R_d|=\varphi(\frac{n}{d})$. 

Furthermore, any $\lq\lq$arithmetic" progression, $(a,ab,ab^2)$, (written multiplicatively) in the group $(\Z/\frac{n}{d}\Z)^\times$ corresponds to the $\lq\lq$geometric" progression $(da,dab,dab^2)$ contained in the set of residues $R_d$.  So any geometric-progression-free subset of $R_d$ cannot be larger than $D((\Z/\frac{n}{d}\Z)^\times)$.  Because the $R_d$ partition $\Z/n\Z$, we see that 
\begin{align}E(\Z/n\Z) \leq \sum_{d\mid n} D((\Z/d\Z)^\times) &= \sum_{\substack{d\mid n\\ d<\sqrt{n}}}D((\Z/d\Z)^\times) + \sum_{\substack{d\mid n\\ d\geq \sqrt{n}}}D((\Z/d\Z)^\times) \nonumber \\
& \ll \sum_{\substack{d\mid n\\ d<\sqrt{n}}}\sqrt{n} + \sum_{\substack{d\mid n\\ d\geq \sqrt{n}}}\frac{\varphi(d)(\log \log d)^5}{(\log d)^{\frac{1}{2}}} \label{levineq} \\
& < \sum_{\substack{d\mid n}}\sqrt{n} + \frac{(\log\log n)^5}{(\frac{1}{2}\log n)^{\frac{1}{2}}} \sum_{\substack{d\mid n}}\varphi(d) \nonumber \\
& \ll n^{1/2+\epsilon}+ \frac{n(\log \log n)^5}{(\log n)^{\frac{1}{2}}} \ll \frac{n(\log \log n)^5}{(\log n)^{\frac{1}{2}}}. \nonumber
\end{align}
Where we used the fact that the number of divisors of $n$ is $O(n^\epsilon)$ for every $\epsilon >0$.
\end{proof}

Bounding the exponent on $\log n$ in this theorem is complicated by the fact that the size, $\lambda(n)$, of the largest cyclic subgroup of the unit group $(\Z/n\Z)^\times$ can occasionally be very small.  In general, however, this is not the case. In fact, we know \cite[Lemma 2]{PomLambda} that there exists a set $S$ of integers with asymptotic density 1 such that for all $n$ in $S$, $\lambda(n) = n/(\log n)^{\log \log \log n + O(1)}$.  

Let $n \in S$ and let $H$ be a cyclic subgroup of $(\Z/n\Z)^\times$ of size $\lambda(n)$.  Then, using Sanders bound \eqref{Sanders} and Proposition \ref{subgp}
\begin{align*}
D((\Z/n\Z)^\times) \leq & \frac{|(\Z/n\Z)^\times|D(H)}{|H|} \\
\ll & \frac{\varphi(N) |H| (\log \log |H|)^5}{|H| \log |H|} \\
= & \frac{\varphi(n) (\log \log (n/(\log n)^{\log \log \log n + O(1)}))^5}{\log (n/(\log n)^{\log \log \log n + O(1)})} \\
\ll & \frac{\varphi(n)(\log \log n)^5}{\log n}.
\end{align*} 
We know \cite[Lemma~2]{Pom123} that for each $d\mid n$, \[\frac{\lambda(d)}{d} \geq \frac{\lambda(n)}{n},\] so for $n \in S$, and $d|n$, $\lambda(d) \geq \frac{d}{(\log n)^{\log \log \log n + O(1)}}$.  If we assume $d>\sqrt{n}$ then the above argument shows that $D((\Z/d\Z)^\times) \ll \frac{\varphi(d)(\log \log d)^5}{\log d}$. Using this inequality  in place of Corollary \ref{unit} in line \eqref{levineq} of the proof above, we have
\begin{theorem} For $n \in S$, $E(\Z/n\Z) \ll \frac{n(\log \log n)^5}{\log n}$.  
\end{theorem} 

\section{Geometric Progressions with Integer Ratio} \label{integer}
As mentioned before, it is possible to consider geometric progressions of two different types, depending on whether or not the ratio common to the progression is an integer.  One would expect that restricting to the integer case should allow us to construct sets with larger asymptotic density, since there are less restrictions on our set.  While Rankin's set, $G^*_3$ is constructed to avoid rational geometric progressions, no integers can be added to it without introducing an integer geometric progression either.  Nevertheless, we can construct sets with substantially higher upper density.  The following construction extends the method described in \cite[Theorem 3.2]{BBHS}.
\begin{theorem} \label{intup} It is possible to construct sets of integers free of geometric progressions (with integer ratio) with upper density greater than $0.815509$.  Using the notation introduced at the end of Section 1, this implies that $\overline{\beta} >0.815509$.
\end{theorem}
\begin{proof} Note that for any $N$, the set $\left(\frac{N}{4},N\right]$ is free of geometric progressions (with integer ratio) since for any $n \in \left(\frac{N}{4},N\right]$ and $r\geq 2$, we have $nr^2 >N$.  This observation, combined with the argument described below, could be used to construct a set with upper density $\frac{3}{4}$.  We can, however, do better.  

Rather than just using the range $\left(\frac{N}{4},N\right]$, we note that the set $\left(\frac{N}{9},\frac{N}{8}\right] \cup \left(\frac{N}{4},N\right]$ also has the property of being free of geometric progressions since for any $n \in \left(\frac{N}{9},\frac{N}{8}\right]$, the integer $2n$ lies in the omitted interval $\left(\frac{N}{8},\frac{N}{4}\right]$, hence $n$ is not part of a geometric progression with common ratio 2, and $9n >N$, meaning $n$ cannot be part of a progression of common ratio greater than or equal to 3.  One can further check that the set 
\[ \mathbb{S}_N = \left(\frac{N}{48},\frac{N}{45}\right] \bigcup \left(\frac{N}{40},\frac{N}{36}\right] \bigcup \left(\frac{N}{32},\frac{N}{27}\right] \bigcup \left(\frac{N}{24},\frac{N}{12}\right] \bigcup \left(\frac{N}{9},\frac{N}{8}\right] \bigcup \left(\frac{N}{4},N\right]\]
has this property, and contains $3523N/4320 > 0.815509N$ integers less than N.  (We can continue this process, adding smaller and smaller intervals to $\mathbb{S}_N$ indefinitely, and create sets with marginally greater density. However, this set is remarkably close to being optimal, we can't take another such interval until $N/2208.$)

Now, fix $N = N_1$, and let $N_2 = 48^2 N_1$.  The set $\mathbb{S}_{N_1} \cup \mathbb{S}_{N_2}$ will also be free of geometric progressions with integer ratio, since if $n,nr \in \mathbb{S}_{N_1}$ then $r<48$ so $nr^2 < 48 N_1 = N_2/48$, and so $nr^2 \not \in \mathbb{S}_{N_2}$. Similarly, if $nr,nr^2 \in \mathbb{S}_{N_2}$, we again have $r<48$ and so 
\[n>\frac{nr}{48}>\frac{N_2}{48^2} = N_1\] thus $n \not \in \mathbb{S}_{N_1}$.  In general, if we set \[N_i = \frac{48^2 N_{i-1}^2}{N_1}\] then no geometric progression with two elements contained in $\mathbb{S}_{N_i}$ will also have an element in the union of the $\mathbb{S}_{N_j}$ with $j <i$, nor vice versa.   
Thus, letting $\mathcal{S}$ be the union of all such $\mathbb{S}_{N_i}$, we find that $\bar{d}(\mathcal{S})>0.8155$, and the entire set $\mathcal{S}$ is free of geometric progressions with integer ratio.    
\end{proof}

We are also able to construct sets free of geometric progressions with integer ratio with a (slightly) higher asymptotic density than the set generated by the greedy algorithm, described by Rankin.  For convenience we recall the proof of Theorem \ref{rankinconstruction}, which constructs Rankin's set, $G^*_3$.
\begin{proof}
Note that if $(a,ak,ak^2)$ is a geometric progression and we denote by $v_p(a)$ the $p$-adic valuation of $a$, then $(v_p(a),v_p(ak),v_p(ak^2))$ forms an arithmetic progression (which is non-trivial if $v_p(k) \neq 0$).  Thus, any set of integers, all of whose prime factors occur with exponent contained in a set $A$ free of arithmetic progressions, will be free of geometric progressions \cite[Theorem 1]{BG}. Take $A= A^*_3=\{0,1,3,4,9,\cdots\}$ to be the set of integers which do not have a digit two in their ternary expansions.  This is the set obtained by choosing integers free of arithmetic progressions using a greedy algorithm. Now, letting $G^*_3 = \{n \in \mathbb{N} : \text{for all primes }p, v_p(n) \in A^*_3 \}$,  we obtain a set free of geometric progressions.  (Note that this set is also the set obtained using a greedy algorithm to choose integers free of geometric progressions, either of integer or rational ratio.)

The density of this set, $G^*_3$, can be found using an Euler product.  The probability that a given integer is divisible by an acceptable power of the prime $p$ is given by 
\begin{align*}\left(\frac{p-1}{p}\right)\left(\sum_{i \in A^*_3}\frac{1}{p^i}\right) &= \left(\frac{p-1}{p}\right)\left(1+\frac{1}{p}\right)\prod_{i>0}\left(1+\frac{1}{p^{3^i}}\right)\\
&=\left(1-\frac{1}{p^2}\right)\prod_{i>0}\left(1+\frac{1}{p^{3^i}}\right)\end{align*}
and since divisibility by different primes is independent,
\[d(G^*_3) =  \prod_p \left(\left({1-\frac{1}{p^2}}\right)\prod_{i>1}\left({1+\frac{1}{p^{3^i}}}\right)\right) =  \frac{1}{\zeta(2)}\prod_{i>0}\frac{\zeta(3^i)}{\zeta(2\cdot 3^i)}>0.7197. \qedhere\]

\end{proof}

So, for example, considering just the primes 2, 3 and 5, Rankin's construction would include numbers of the form $2{\cdot}3{\cdot}5k$, where $k$ is an integer divisible only by primes larger than 5, and not to any powers that it would cause it to be otherwise excluded from the set. If we exclude these integers instead, however, we would be able to include numbers of the following forms:
\[2^2{\cdot}3{\cdot}5k, \hspace{1.5mm} 2{\cdot}3^2{\cdot}5k, \hspace{1.5mm} 2{\cdot}3{\cdot}5^2k, \hspace{1.5mm} 2^2{\cdot}3^2{\cdot}5k, \hspace{1.5mm} 2{\cdot}3^2{\cdot}5^2k, \hspace{1.5mm} 2^2{\cdot}3{\cdot}5^2k \hspace{0.5mm} \text{ and } \hspace{0.5mm}2^2{\cdot}3^2{\cdot}5^2k.\]
Each of these new numbers we've included will force us to exclude numbers with these primes to higher powers, but in the end (calculating these inclusion/exclusions with a computer) we find that we gain about 0.0022 asymptotically in the trade, producing a new set with asymptotic density 0.72195.  This proves a new lower bound for the constant $\beta$, defined in Section \ref{background}.
\begin{theorem} \label{betalow} We have the lower bound $\beta > 0.72195$.
\end{theorem}
 This can be further improved by incorporating more primes and exclusions.  Note, however, that this larger set does include progressions with rational ratios.  For example, progressions of the form $(2{\cdot}3^2 {\cdot}5k, 2^2 {\cdot}3 {\cdot}5k, 2^3 {\cdot}5k)$ are included even though they form a progression with common ratio $\frac{2}{3}$.

\section{Upper Bounds} \label{bounds}

Upper bounds for the densities of sets of integers free of geometric progressions have been studied in several papers, with the current best bound being that of Nathanson and O'Bryant that the upper density of any geometric progression free set is at most 0.84948. (Riddell \cite{Riddell} gives the upper bound 0.8399, but states that $\lq\lq$The details are too lengthy to be included here.") In this section we improve this bound to 0.772059.

We consider first the problem of avoiding geometric progressions with ratios involving only a finite set of primes, in particular the primes smaller than some bound $s$.  Denote by $g_s(N)$ the cardinality of the largest subset of the integers $\{1,\cdots, N\}$ which is free of 3-term rational geometric progressions which have common ratio involving only the primes less than or equal to $s$.  We will see that for any $s$ the limit $\lim_{N \to \infty} \frac{g_s(N)}{N}$ exists, which we will denote by $\overline{\alpha_s}$.  We first consider the specific case of just the primes 2 and 3.

\begin{theorem} \label{23bound} The limit $\overline{\alpha_3} =\lim_{N \to \infty} \frac{g_3(N)}{N}$ exists and is bounded by \[0.790470 < \overline{\alpha_3} < 0.791266.\] 
\end{theorem}
\begin{proof}
Fix $N>0$ and consider the largest subset of the integers $\{1,\cdots, N\}$ free of geometric progressions which have a common ratio involving only the primes 2 and 3.  Denote by $S^3_k$ the set of $3$-smooth numbers (numbers whose only prime divisors are 2 and 3) at most $k$.  Note first that any geometric progression free subset of $S^3_4 = \{1,2,3,4\}$  must exclude at least one integer from this set, and hence for any integer $b\leq \frac{N}{4}$ such that $(b,6)=1$ our set must exclude at least one of the integers $b,2b,3b,4b$.  (The single one excluded cannot be $3b$.)  Since these numbers are distinct for different values of $b$, we must exclude a total of at least $\frac{1}{3}\left(\frac{N}{4}\right) + O(1)$ integers.  

If we now consider $S^3_9 = \{1,2,3,4,6,8,9\}$, we find that this set contains the 4 progressions $(1,2,4),(2,4,8),(1,3,9)$ and $(4,6,9)$ which cannot all be precluded by removing any single number.  However, removing the two integers $2$ and $9$ suffices.  This means that for each $b\leq \frac{N}{9}$, $(b,6)=1$, we must exclude at least two of the integers from the set $\{b,2b,3b,4b,6b,8b,9b\}$, and moreover these sets are disjoint, not only from each other, but also simply extend the sets constructed from $S^3_4$ above.  Thus each $b\leq \frac{N}{9}$ with $(b,6)=1$ corresponds to an additional excluded integer, meaning we must now exclude at least $\frac{1}{3}\left(\frac{N}{4} +\frac{N}{9}\right)+O(1)$ integers.  In general, each time the largest geometric-progression free subset of $S^3_k$ requires an additional exclusion, there are an additional $\frac{N}{3k}$ integers which must be excluded from our set.  One can check computationally that the first few values of $k$ which require an additional exclusion are given in the following table.
\begin{table}[H]
\caption{}
\begin{center}
\begin{tabular}{|c|c|c|c|c|c|c|c|} 
\hline 
$k$ & \# of integers & & $k$ & \# of integers& & $k$ &\# of integers\\
& excluded from $S^3_k$ & & & excluded from $S^3_k$& & & excluded from $S^3_k$\\ 
\hline 
4 & 1 & &243 & 13 & & 1458 &25\\
9 & 2 & &256 & 14 & & 1728 &26\\
16 & 3 & &288 & 15 & & 1944 &27\\ 
18 & 4 & &384 & 16 & & 2048 &28\\ 
32 & 5 & &486 & 17 & & 2304 &29\\
36 & 6 & &512 & 18 & & 2592 &30\\ 
64 & 7 & &576 & 19 & & 3072 &31\\
81 & 8 & &729 & 20 & & 3888 &32\\
96 & 9 & &864 & 21 & & 4096 &33\\
128& 10& &972 & 22 & & 4374 &34\\
144& 11& &1024& 23 & & 5184 &35\\
192& 12& &1296& 24 & & 5832 &36\\ 
\hline  
\end{tabular} 
\end{center}
\end{table}
Taking all these exclusions into account, we find that we've excluded 
\[\frac{N}{3}\left(\frac{1}{4}{+}\frac{1}{9}{+}\frac{1}{16}{+}\frac{1}{18}{+}\frac{1}{32}{+}\cdots{+}\frac{1}{5832}\right) +O(1)\] 
numbers.  Since $\frac{1}{3}\left(\frac{1}{4}{+}\frac{1}{9}{+}\cdots{+}\frac{1}{5832}\right) > 0.208734$, we have that any subset of $\{1,\cdots,N\}$ free of geometric progressions with 3-smooth ratios has size at most $0.791266N$ for $N$ sufficiently large.

Note that the process described above is also constructive: For a fixed integer $N$  we can take for each $b \leq N$, $(b,6)=1$ the set of integers not excluded above (also excluding multiples by 3-smooth numbers that we have not yet taken into account) and obtain a subset of the integers up to $N$ free of progressions involving the primes 2, and 3.  This set we construct will differ in size from our upper bound only by the trailing terms in the series of 3-smooth numbers that we have not yet taken into account, 
\[\sum_{\substack{n>5832 \\ n \text{ is 3-smooth}}}\frac{1}{3n} < 0.000795, \]
and so can be taken to be at least $0.790470N$. 
\end{proof}

We can actually take this further.  Using the methods of Theorem \ref{intup} we can extend the  construction above to a set of integers which has this upper density while avoiding 2,3-rational progressions.  This gives us a lower bound for the supremum of the upper densities of all sets that avoid rational progressions involving only the primes 2 and 3.  Since the upper and lower bounds in this argument differ only by the trailing terms in the series of reciprocals of 3-smooth numbers, the series we are computing will converge to the actual supremum over the upper densities of all sets avoiding such progressions.  

There is nothing particular about the primes 2 and 3 in this argument.  In general, write the sequence of $s$-smooth numbers
$ 1 = n_1 < n_2 < \cdots $ in increasing order.  For each $j$ let $m_j$ denote the size of the largest subset of $S^s_{n_j} = \{n_1,n_2,\cdots,n_j\}$ which has no triples in geometric progression.  Note that the $m_j$ are nondecreasing.  Let $I_s$ be the set of numbers $j$ with $m_j = m_{j-1}$ (those $j$ for which $S^s_{n_j}$ requires an additional exclusion).  For $s=3$ these are the numbers appearing in Table 1.  We have the following result.
\begin{theorem} \label{alphas} For each integer $s\geq 2$, 
\[\overline{\alpha_s} = 1-\left(\prod_{p\leq s} \frac{p-1}{p}\right)  \sum_{j\in I_s}  \frac{1}{n_j}.\] Furthermore,
\[\overline{\alpha_s} = \sup\{\bar{d}(A):A \subset \N \text{ is free of }s\text{-smooth rational geometric progressions}\}.\]
\end{theorem}

Since any geometric-progression-free subset of the integers must, in particular, be free of ratios involving only the primes 2 and 3, we see that the upper bound above, 0.791266, is also an upper bound for $\overline{\alpha}$.  Already this value is better than the bounds given before in the literature, but we can improve this result further. 

If we consider now the primes 2,3 and 5, we see that the proof above goes through in exactly the same way, requiring additional exclusions at each of the integers
\begin{equation}\label{235ex} 
\begin{tabular}{r r r r r r r r r r r r r r r r r r}
4 & 9 & 16 & 18 & 20 & 25 & 32 & 36 & 50 & 60 & 64 & 75 & 80\\
96 & 100 & 108 & 128 & 144 & 150 & 160 & 192 & 200 & 225 & 240 & 243 & 256\\
300 & 320 & 324 & 384 & 400 & 432 &	480 & 500 & 512 & 540 & & & 
\end{tabular} \\
\end{equation} 
which are the first 36 exclusions required. This list gives us the bounds $0.766512 <  \overline{\alpha_5} < 0.782571$.  The difficulty in pushing this method further is the amount of computation required to find the largest geometric-progression-free subset of $S^5_k$ (the 5-smooth numbers up to $k$) for increasingly larger $k$.  For example, showing that 576 (the next 5-smooth number after 540) requires an additional exclusion would require showing that there are no geometric-progression-free subsets of size 36 among the 70 5-smooth numbers up to 576.

Even though computational limitations prevent us from finding the exact values where additional exclusions are necessary past 540 in $S^5_k$ we can still use some of the computational work we did to estimate $\overline{\alpha_3}$ to further improve the upper bound on $\overline{\alpha_5}$.  Just by considering the 2,3-ratios among the 5-smooth numbers, we see that each time we multiply the numbers from Table 1 (the integers where an additional exclusion was required in the 3-smooth case) by successive powers of 5 we obtain a list of integers which are each an upper bound for when an additional exclusion must be made in the 2,3,5-ratio case.  So, for example, multiplying each of $4, 9, 16, 18, 32 \ldots 5832$ by $1, 5, 25, 125 \ldots$ and reordering we obtain the list:  
\begin{equation} \label{235excalc}
\begin{tabular}{r r r r r r r r r r r r r r r r r r}
4 & 9 & 16 & 18 & 20 & 32 & 36 & 45 & 64 & 80 & 81 & 90 & 96\\
100 & 128 & 144 & 160 & 180 & 192 & 225 & 243 & 256 & 288 & 320 & 384 & 400\\
405 & 450 & 480 & 486 & 500 & 512 & 576 & 640 & 720 & 729 & 800 & 864 & 900 &\ldots
\end{tabular} \\
\end{equation}
Note that each term in this list \eqref{235excalc} is greater than or equal to the corresponding term in \eqref{235ex}.  Looking at the 37th entry of this table we see that we will require an additional exclusion by the time we reach 800.  So, taking all of these exclusions into account (first the 36 exclusions from \eqref{235ex}, and then those starting at 800 from \eqref{235excalc})  we can decrease our bound by an additional 0.006815, so $\overline{\alpha_5} < 0.775755$. Applying this process again for the primes 2,3,5 and 7, where we compute that exclusions must be made at 
\begin{equation}
\begin{tabular}{r r r r r r r r r r r r r r r r r}
4 & 9 & 16 & 18 & 20 & 25 & 28 & 32 & 36 & 49 & 50 & 60 & 64\\
72 & 75 & 81 & 96 & 98 & 100 & 108 & 112 & 126 & 128 & 144 & 147 & 150
\end{tabular} \\
\end{equation}
and incorporating the exclusions calculated for both 2,3 and for 2,3,5, as described above, we obtain the bound $\overline{\alpha_7} < 0.772059$. Again, this is also an upper bound for the upper density of a set of integers avoiding all geometric progressions which proves the following.

\begin{theorem} We have $\overline{\alpha} < 0.772059$.
\end{theorem}

Since this upper bound is lower than the upper density of the set we constructed for the integer-ratio problem in Theorem \ref{intup}, we see (as one might have expected) that these two problems, considering integer and rational ratios, are in fact different.  One can carry through an analogous argument considering only progressions with integer ratios, in which case we find (looking at 3-smooth integer progressions) that we must make exclusions at
\begin{center}
\begin{tabular}{r r r r r r r r r r r r r r r r r r}
4 & 9 & 18 & 32 & 48 & 64 & 96 & 128 & 144 & 192 & 256 & 288 & 384\\
432 & 512 & 648 & 864 & 972 & 1024 & 1296 & 1536 & 1944 & 2187 & 2304 & 2916 & 3456\\
4096 & 4608 & 5832 & 6144 & 6912 & 8748 & 9216
\end{tabular} \\
\end{center}
yielding an upper bound of 0.820555.  If we combine this argument, as in the rational case above, with the necessary exclusions for 5-smooth progressions, 
\begin{center}
\begin{tabular}{r r r r r r r r r r r r r r r r r}
4 & 9 & 18 & 20 & 32 & 40 & 48 & 64 & 80 & 96 & 100 & 128 & 144\\
160 & 192 & 200 & 240 & 256 & 288 & 320 & 384 & 400 & 432 & 400 & 432 & 480\\
500 & 512\\
\end{tabular} 
\end{center}
we can compute a new upper bound for $\overline{\beta}$ which is less than 0.004 above the lower bound of the set we constructed in section \ref{integer}. 
\begin{theorem}We have $0.815509 < \overline{\beta} < 0.819222$.
\end{theorem}

We can also use the set we constructed in the proof of Theorem \ref{23bound}, which had high upper density while avoiding geometric progressions involving the primes 2 and 3, to construct sets free of any rational ratio progression and have a higher upper density than Rankin's set.  

\begin{theorem} There exist geometric-progression-free sets with upper density greater than $0.730027$, so $0.730027 < \overline{\alpha} < 0.772059$. \label{lowerrat}
\end{theorem}
\begin{proof}Recall that Rankin's construction consisted of integers with exponents on primes contained in the set $A_3^*$, the greedily chosen set free of arithmetic progressions.  To construct a set with greater upper density, we start with the set described above, in Theorem \ref{23bound} free of geometric progressions involving the primes 2 and 3 and with upper density 0.790470, and then remove from it those integers which have a prime greater than 3 with exponent not contained in $A_3^*$.  Essentially, rather than taking all integers, $b$, coprime to 6 in the argument above, we use Rankin's construction to choose integers coprime to 6 which do not themselves contain any geometric progressions.  We then have found a more efficient way (in regard to upper density) of choosing exponents for the primes 2 and 3 than Rankin's method. 

To find the upper density of this set, we recall the Euler product of the density of Rankin's set, 
\[ \frac{1}{\zeta(2)}\prod_{i>0}\frac{\zeta(3^i)}{\zeta(2\cdot 3^i)} = \prod_p\left(\frac{p-1}{p}\sum_{i \in A_3^*}p^{-i}\right).\]
Now, fix an integer $N$, and consider, for example, the interval $[N/6,N/8)$.  Since 4 required one exclusion and 6 did not yet require an additional exclusion, we could take for each $b \in [N/6,N/8)$ with $(b,6)=1$ four of the five integers $b,2b,3b,4b,6b$ without creating a 2,3-rational progression, a total contribution of $\frac{4}{3}(\frac{N}{4}-\frac{N}{6})+O(1)$ integers less than $N$.  We now add the further restriction that our integers $b$ must be in the set $G^*_3$ (In order to avoid progressions involving other primes) as well as being coprime to 6.  Such candidates for $b$ have asymptotic density 
\[\frac{1}{3}\prod_{p\geq 5}\left(\frac{p-1}{p}\sum_{i \in A_3^*}p^{-i}\right)\]
and so the contribution from the range $[N/6,N/8)$ is now 
\[\frac{4}{3}\left(\frac{N}{4}-\frac{N}{6}\right)\prod_{p\geq 5}\left(\frac{p-1}{p}\sum_{i \in A_3^*}p^{-i}\right)+O(1).\]
Doing this for every interval gives us a contribution of 
\[\left(1-\frac{N}{3}\left(\frac{1}{4}{+}\frac{1}{9}{+}\cdots{+}\frac{1}{5832}+\sum_{\substack{n>5832 \\ n \text{ is 3-smooth}}}\frac{1}{3n}\right)\right)\prod_{p\geq 5}\left(\frac{p-1}{p}\sum_{i \in A_3^*}p^{-i}\right)+O(1).\]
As in the proof of Theorem \ref{intup} we can stitch together this construction for increasingly larger values of $N$, yielding a set with upper density greater than
\[0.790470\prod_{p\geq 5}\left(\frac{p-1}{p}\sum_{i \in A_3^*}p^{-i}\right) > 0.730027. \qedhere\]
\end{proof}

The specific case of sets of integers which avoid progressions involving only a single prime (or in fact any single integer) was recently studied by Nathanson and O'Bryant \cite{irrational}, in which they find that the upper density described above converges to an irrational number.  In the case of progressions involving only the prime 2, the series for $\overline{\alpha_2}$ converges to an irrational number approximately $0.846378$ with error less than $0.000001$. 

\section{Computing $\overline{\alpha}$}
The arguments given above show, not only how to compute good approximations for each $\overline{\alpha_s}$ which we can use to bound $\overline{\alpha}$, but also that these values converge to $\overline{\alpha}$.
\begin{theorem}
In the limit, as a larger set of initial primes is taken into account, $\lim_{s \to \infty} \overline{\alpha_s} = \overline{\alpha}$.
\end{theorem}
\begin{proof}
The arguments of Theorems \ref{lowerrat} and \ref{alphas} show that \[\overline{\alpha_s}\prod_{p>s}\left(\frac{p-1}{p}\sum_{i \in A_3^*}p^{-i}\right)\leq \overline{\alpha}\leq \overline{\alpha_s}\]
and since this Euler product converges, we know that \[
\lim_{s \to \infty} \prod_{p>s}\left(\frac{p-1}{p}\sum_{i \in A_3^*}p^{-i}\right) = 1.\]
The conclusion follows.
\end{proof}

Thus, since we have a method to compute each $\overline{\alpha_s}$ to any desired precision, we can also do so for the constant $\overline{\alpha}$ by extending the methods described above.  We look here at the complexity of computing $\overline{\alpha}$.
\begin{theorem}
For each number $\epsilon$ with $0<\epsilon<1$, the constant $\overline{\alpha}$ can be computed to within $\epsilon$ in time $O\left(1.6538^{(-2\log_2 \epsilon)^{\frac{1}{\epsilon}}}\right)$.
\end{theorem}
\begin{proof}
We need, first, to consider a sufficient number of primes so that \[\prod_{p>s}\left(\frac{p-1}{p}\sum_{i \in A_3^*}p^{-i}\right) > 1- \epsilon/2.\]  
Where $A^*_3$ is greedily chosen set of integers free of arithmetic progressions used in Rankin's construction.  Using the inequality
\begin{align*}
\prod_{p>s}\left(\frac{p-1}{p}\sum_{i \in A_3^*}p^{-i}\right) &> \prod_{p>s}\frac{p-1}{p}\left(1+\frac{1}{p}\right) =\prod_{p>s}\left(1-\frac{1}{p^2}\right)\\
&>\prod_{n>s}\left(1-\frac{1}{n^2}\right) = \prod_{n>s}\frac{(n-1)(n+1)}{n^2} = \frac{s}{s+1},
\end{align*}
we see that taking $s > \frac{1-\epsilon/2}{\epsilon/2}=\frac{2-\epsilon}{\epsilon}$ suffices.  

Second, we need to compute $\overline{\alpha_s}$ to sufficient accuracy so that the trailing terms in our series of reciprocals of $s$-smooth numbers is less than $\epsilon/2$.  Let $N = \left\lceil-\log_2\left(\frac{\epsilon}{2\pi(s)}\right)\right\rceil$.  
Then \[\sum_{i>N} \frac{1}{p^i}\leq \sum_{i>N} \frac{1}{2^i} = \frac{1}{2^N} \leq \frac{\epsilon}{2\pi(s)}.\]  Then the error in approximating $\overline{\alpha_s}$ by using the $(N+1)^{\pi(s)}$ $s$-smooth integers $B^s_N = \{ 2^{i_1}3^{i_2}\cdots p_{\pi(s)}^{i_{\pi(s)}}:0\leq i_k \leq N \}$ is less than 
\begin{align*}
\left(\prod_{p\leq s}\frac{p-1}{p}\right)\sum_{\substack{n\text{ is $s$-smooth}\\ n \not \in B^s_N}} \frac{1}{n} &< \left(\prod_{p\leq s}\frac{p-1}{p}\right)\sum_{p\leq s}\sum_{\substack{n\text{ is $s$-smooth}\\ v_p(n)>N}} \frac{1}{n} \\
& = \left(\prod_{p\leq s}\frac{p-1}{p}\right)\sum_{p\leq s}\left(\left(\sum_{i>N}\frac{1}{p^i} \right)\prod_{\substack{q \neq p\\ q\leq s}}\sum_{i\geq 0}\frac{1}{q}\right)\\
&= \sum_{p\leq s}\left(\frac{p-1}{p}\left(\sum_{i>N}\frac{1}{p^i} \right)\right)< \sum_{p\leq s} \frac{\epsilon}{2\pi(s)} = \frac{\epsilon}{2}.
\end{align*}
Now, using the $(N+1)^{\pi(s)}$ smallest $s$-smooth integers, rather than those in $B^s_N$ will only make the error smaller.  So it will suffice for our computation to work with the exclusions required among the first $(N+1)^{\pi(s)}$ $s$-smooth integers.  In particular, we need to calculate for each $j \leq K$ the minimal number of exclusions required from the set of integers $S^s_j$.  

We need to exclude at least one member of each 3-term geometric progression contained in this set, an example of a 3-hitting set problem, which is a problem known to be NP-complete.  In \cite{wahlstrom} Wahlstr\"{o}m gives an algorithm for computing a 3-hitting set for a set of size $n$ (and any collection of 3-element subsets of it) in time $O\left(1.6538^n\right)$.  (He also gives an algorithm that requires exponential space but runs in time $O\left(1.6318^n\right)$.)

Using this algorithm for each $j \leq K$ to compute the minimal number of exclusions will require  time $O\left(\sum_{j\leq K} 1.6538^j\right) = O\left(1.6538^K\right)$ in total.  Substituting in the definitions of $K$ and $N$, we see this takes time 
\[O\left(1.6538^{\left(\left\lceil-\log_2\left(\frac{\epsilon}{2\pi\left(\frac{2-\epsilon}{\epsilon}\right)}\right)\right\rceil + 1\right)^{\pi\left(\frac{2-\epsilon}{\epsilon}\right)}}\right) = O\left(1.6538^{\log_2\left(\frac{8\pi\left(\frac{2-\epsilon}{\epsilon}\right)}{\epsilon}\right)^{\pi\left(\frac{2-\epsilon}{\epsilon}\right)}}\right)\]
or, using the crude inequality  $\pi\left(\frac{2-\epsilon}{\epsilon}\right) < 8\pi\left(\frac{2-\epsilon}{\epsilon}\right) < \frac{1}{\epsilon}$ for sufficiently small $\epsilon$, \[O\left(1.6538^{(-2\log_2 \epsilon)^{\frac{1}{\epsilon}}}\right).\]
Having done so, and letting $M$ be the calculated partial sum of reciprocals that require an additional exclusion times $\prod_{p\leq s}\frac{p-1}{p}$ we see that 
\begin{align*}
1-M>\overline{\alpha} &> \left(1-M -\frac{\epsilon}{2}\right)\prod_{p>s}\left(\frac{p-1}{p}\sum_{i \in A_3^*}p^{-i}\right)\\
 &> \left(1-M -\frac{\epsilon}{2}\right)\left(1-\frac{\epsilon}{2}\right)>1-M-\epsilon
\end{align*}
and thus we have achieved the required precision in our estimate of $\overline{\alpha}$.

\end{proof}
The same arguments apply for the constant $\overline{\beta}$ as well.  While it appears from our computations so far that the argument of Theorem \ref{intup} is far more efficient at computing lower bounds for $\overline{\beta}$ than the analogous argument to that described here for integer-ratios, we cannot prove that the construction described there converges to $\overline{\beta}$.

\section{Sets with Asymptotic Density} \label{density}
All of the upper bounds given here and elsewhere in the literature are for the upper density of geometric-progression-free sets, while the set Rankin constructed to produce a lower bound has an asymptotic density.  It is quite possible that the restricted collection of such sets which possess an asymptotic density will have smaller densities.  

We can prove that this is the case when avoiding progressions involving one prime if we slightly strengthen the hypothesis that our sets have an asymptotic density.  We say that a set $S$ has a 2-graded density if every subset $S_i = \{s \in S:v_2(s) = i\}$ has an asymptotic density.  Any set with a 2-graded density also has an asymptotic density. We show here that if $S$ is a set free of geometric progressions involving powers of 2 and has a 2-graded density then $S$ will have asymptotic density strictly smaller than $\overline{\alpha_2} \approx 0.846378$, the supremum of the upper densities of all 2-smooth geometric-progression-free sets mentioned at the end of Section \ref{bounds}.

\begin{theorem}The set $T = \{n \in \mathbb{N}:v_2(n) \in A^*_3\}$ (where $A^*_3$ is the set free of arithmetic progressions obtained by the greedy algorithm) has the largest asymptotic density among all sets $S$ which are free of $2$-geometric progressions (of length $3$) and have a $2$-graded density.  This set $T$ has asymptotic density $d(T) < 0.845398 < \overline{\alpha_2}$.  
\end{theorem}
\begin{proof}  The set $T$ is free of progressions involving powers of 2 and for each $i$, 
\[d(T_i)= \left\{ \begin{array}{ll}
            2^{-i-1} & \quad i \in A^*_3 \\
            0 & \quad i \not \in A^*_3
        \end{array}.
    \right. \]
Furthermore, $T$ is the set obtained by the greedy algorithm, each integer $t$ is included in $T$ if doing so does not create a 2-smooth geometric progression with smaller integers already in $T$.  (Note that $d(T) = \frac{1}{2}\sum_{i \in A^*_3}2^{-i} <  0.845398$.)

Suppose that $T$ were not the optimal set with these properties.  Namely, let $T'$ be a 2-graded set free of 2-smooth geometric progressions, with a greater density than $T$.  Because $T'$ is 2-graded we have $d(T')=\sum_{i} d(T'_i)$, and since $d(T')>d(T)$  there must exist some $i \not \in A^*_3$ such that $d(T'_i) >0$. In other words, numbers containing one of the powers of 2 forbidden from appearing in the greedy set $T$ must make a positive contribution to the density of $T'$.  Now, because each integer $t \in T'_i$ is not in the set $T$ which was obtained by the greedy algorithm, it must be precluded by the existence of at least  one smaller integer $u \in T$, $u \not \in T'$ and where $t = 2^ku$ for some integer $k$. 

As there are infinitely many such integers $t \in T'_i$, and each corresponds to at least one distinct $u < t$ with the property above, we have by the pigeon hole principle that there must be at least one $j \in A^*_3$, $j<i$ such that $d(T'_j)<2^{-j-1}= d(T_j)$.  For such $j$ the density contribution of the set $T'_i$ is strictly less than the the corresponding contribution of $T_j$ to $T$.  Let $j$ be the least such integer with this property, and let $U = \{u : v_2(u) = j, u \not \in T'_j\}$. Then $d(U) = 2^{-j-1}-d(T'_j) = \delta \geq 0$.  Essentially, the elements of $U$ have been excluded from $T'$ in order to include certain integers later which have greater powers of 2.  We will see, however, that this trade is not optimal.

Let $U' = \{2^k u:u \in U, k \geq 1\}$ be all of the multiples of elements of $u$ times a power of 2,  which we can consider to be all of the integers whose containment in $T'$ could possibly be affected by including instead the elements of $U$.  Then $d(U') = \sum_{i\geq 1} \delta/2^{i} = \delta$.  The intersection of $U'$ with $T'$ must be strictly smaller however, as the set $U'$ contains many 2-smooth geometric progressions. So $d(U \cap T') <\delta$, all of the multiples of elements of $U$ by powers of 2 which could possibly be included in a 2-smooth geometric-progression-free set have a smaller density than the elements of $U$ itself.  Thus we can construct a larger 2-smooth geometric-progression-free set by including the elements of $U$ in our set rather than any of the elements of $U'$.  The set $T''=(T'\setminus U') \cup U$ will also be free of 2 geometric progressions, will satisfy the density requirements and $d(T'') = d(T' \setminus U')+d(U) > d(T')-\delta+\delta = d(T')$.  $T''$ contains the elements of $U$ rather than (some of) those from $U'$ and has a strictly larger density.  We can now repeat this process for each $j \not \in A^*_3$, and argue that it is never optimal to have $d(T'_j)>0$.  Therefore, the largest density we can obtain is by taking the maximum possible contribution from each set $T_j$, $j \in A^*_3$, namely the set $T$ itself.
\end{proof}

\section{Open Questions}
This paper answers one of the questions in \cite{NO} by demonstrating a method of effectively computing the value of $\overline{\alpha}$.  The question posed in that paper asks for the maximal upper density of sets avoiding progressions of arbitrary length, $k$.  Our constant, $\overline{\alpha}$, is defined in regards to progressions of length 3, but the methods here easily generalize to progressions of any length, $k$. Their second question, however, regarding the precise value of $\alpha$ remains open.  We do not even know the answer to the following question.

\begin{question} Is $\alpha$ strictly smaller than $\overline{\alpha}$?
\end{question}
This seems almost certain to be true, especially given the result of Section \ref{density}.  In light of this result, and the computations required for Theorem \ref{betalow} we make the stronger conjecture.

\begin{conjecture} Rankin's set, $G^*_3$ has the largest possible density among geometric-progression-free sets which have a density, so $\alpha = d(G^*_3)$.
\end{conjecture}

One can ask the same question about $\beta$.
\begin{question} Is $\beta < \overline{\beta}$?
\end{question}
\begin{question} While we know from Theorem \ref{betalow} that $\beta > d(G^*_3)$, and from Section \ref{bounds} that $\overline{\beta} > \overline{\alpha}$, do we have $\beta > \alpha$?
\end{question}
\begin{question} Do the densities of the sets constructed by extending the method of Theorem \ref{betalow} converge to $\overline{\beta}$?  
\end{question}
\noindent If so, this would likely result in a far more efficient method of computing $\overline{\beta}$. 

In \cite{BBHS}, Beiglb\"{o}ck, Bergelson, Hindman, and Strauss take a more Ramsey-theoretic view of the problem.  One of their questions can be answered with the methods of their paper.  They ask if there is a set $A$ of positive integers free of 3-term rational ratio geometric progressions, such that $A$ has positive upper density and $A$ contains arbitrarily long intervals.  Such a set can be constructed by alternating between long runs from Rankin's set of positive density, long gaps with no integers, and consecutive integers in an interval of the shape $[x,x+\sqrt{x}-1]$. Their Lemma 3.3 implies such an interval of consecutive integers has no 3-term rational-ratio geometric progressions.  Modifying this slightly, the set can even be taken to have positive lower density.  Thanks are due to Carl Pomerance for these observations.  In addition, he observes that it is fairly trivial to obtain a van der Waerden type theorem: For any $k$-coloring of the natural numbers, there are arbitrarily long monochromatic integer ratio geometric progressions.  To see this, consider a second $k$-coloring on the integers, where the color of $j$ is determined by the original coloring of $2^j$, and then apply the original van der Waerden theorem to this second coloring.  Here is a nice problem from [1] that remains unsolved.

\begin{question}  Must every infinite set of natural numbers with bounded
gaps between consecutive terms contain arbitrarily long geometric progressions?
\end{question}

\section*{Acknowledgments}
I would like to thank Kevin Ford for suggesting the problem of geometric progressions in $\Z/n\Z$ and my advisor, Carl Pomerance, for his support and invaluable guidance during the development of this paper.

\bibliographystyle{amsplain}
\bibliography{geometricseq}														
\end{document}